\documentclass[a4paper, 12pt]{article}
\usepackage{amsmath}
\usepackage{amssymb}
\usepackage{amscd}
\usepackage{amsthm}
\usepackage[totalwidth=16.5cm, totalheight=24.5cm]{geometry}
\usepackage{graphicx}
\usepackage{hyperref}    

\newtheoremstyle{boldplain}
{9pt}
{9pt}
{\itshape}
{}
{\bfseries}
{.}
{.5em}
{\thmname{#1}\thmnumber{ #2}\thmnote{ (#3)}}%

\newtheoremstyle{bolddefinition}
{9pt}
{9pt}
{}
{}
{\bfseries}
{.}
{.5em}
{\thmname{#1}\thmnumber{ #2}\thmnote{ (#3)}}%

\theoremstyle{boldplain}

\newtheorem{proposition}[equation]{Proposition}

\newtheorem{theorem}[equation]{Theorem}

\theoremstyle{bolddefinition}

\newtheorem{remark}[equation]{Remark}

\numberwithin{equation}{section}

\newcommand{\R}{{\mathbb R}}

\newcommand{\N}{{\mathbb N}}

\newcommand{\ora}{\overrightarrow}
\newcommand{\pihalf}{\frac{\pi}{2}}
\newcommand{\tits}{\partial_{T}}
\newcommand{\dtits}{\angle_{T}}
\newcommand{\weylch}{\Delta_\text{mod}}

\DeclareMathOperator{\CAT}{CAT}

\DeclareMathOperator{\minset}{Min}
\DeclareMathOperator{\rank}{rank}

\begin{document}

\title{On the displacement function of isometries
of Euclidean buildings}
\author{Carlos Ramos-Cuevas}
\date{February 24, 2014}

\maketitle

\begin{abstract}
In this note we study the displacement function 
$d_g(x):=d(gx,x)$ of an isometry $g$ of a Euclidean building. 
We give a lower bound for $d_g(x)$ depending on the distance
from $x$ to the minimal set of $g$. This answers
a question of Rousseau \cite[4.8]{Rousseau:eximmob} 
and Rapoport-Zink \cite[2.2]{RapoportZink}.

\end{abstract}

\section{Introduction}

Let $(X,d)$ be a metric space and let $g\in Isom(X)$
be an isometry of $X$. The {\em displacement
function} $d_g:X\rightarrow [0,\infty)$ of $g$ is the function
given by $d_g(x):=d(gx,x)$. 
The infimum $\delta_g:=\inf\{d_g(x)\;|\; x\in X\}$ of this function
is called the {\em displacement} or {\em translation length} of $g$. Let us denote with
$\minset_g:=\{x\in X\;|\; d_g(x)=\delta_g\}$ the subset of $X$
where this infimum is attained. We call $\minset_g$ the {\em minimal set} of $g$.
We can divide the isometries of $X$ in three classes depending on the behavior of
their displacement functions:
we say that $g$ is {\em elliptic} if it fixes a point, i.e.\ $\minset_g\neq \emptyset$
and $\delta_g=0$; {\em hyperbolic} or {\em axial}
if $\minset_g\neq \emptyset$ and $\delta_g>0$;
and {\em parabolic} if $\minset_g=\emptyset$. 
We say that $g$ is {\em semisimple} if it is elliptic or hyperbolic.

Suppose now that $(X,d)$ is a complete $\CAT(0)$ space. 
Let $C\subset X$ be a closed convex subset. 
Then the {\em nearest point projection} $p=p_C:X\rightarrow C$ is a well defined
1-Lipschitz map. 
It holds that the angle $\angle_{p(x)}(x,z)$ at $p(x)$ between $x$ and $z$
is $\geq \pihalf$ for all $z\in C$
and if $y\in X$ lies on the segment $xp(x)$ between
$x\in X$ and its projection $p(x)$, then $p(y)=p(x)$ 
(see e.g.\ \cite[Prop.\ II.2.4]{BridsonHaefliger}).
Let $g\in Isom(X)$ be a semisimple isometry,
then the convexity of the distance function implies that the minimal set
$\minset_g$ is a closed
convex subset and the displacement function is an increasing function of the
distance to the minimal set. 
More precisely, let $p:X\rightarrow \minset_g$ 
be the nearest point projection. Suppose $\rho:[0,\infty)\rightarrow X$
 is a geodesic ray with 
$\rho(0)\in \minset_g$ and $p(\rho(t))=\rho(0)$ for all $t$, that is, $\rho$
is orthogonal to $\minset_g$.
It follows from the convexity of the distance function and the 1-Lipschitz property
of the projection that $d_g\circ\rho$ is a  strictly monotonically increasing function.
We are interested in understanding the growth of this function
in the case when $X$ is a Euclidean building. 
An upper bound for $d_g\circ\rho$ given by $\delta_g+2t$ 
follows easily in general from the triangle inequality. 
It is also not difficult to give examples where this bound is attained.
We are therefore interested in giving lower bounds for the function $d_g\circ\rho$.

Euclidean buildings are a special kind of $\CAT(0)$ spaces. 
They are built up from top dimensional building blocks called {\em apartments}
isometric to a Euclidean space. These apartments are glued together
following a pattern given by a Euclidean Coxeter complex (a class of groups
of isometries of the Euclidean space generated by reflections). 
This gluing pattern and a certain angle rigidity in the space of directions at points
forces the geometry of Euclidean buildings to have 
some discreteness nature. 
It follows that in many cases a geometric property at a point or a segment
does not depend on the point or segment in 
question or even on the Euclidean building itself
but only on the type of its Coxeter complex and the 
relative position of the point or segment
with respect to this Coxeter complex.

If $X\cong\R^k$ is a Euclidean space, 
then the function $d_g\circ\rho$ above is given by
$d_g(\rho(t)) = \sqrt{\delta_g^2 + Ct^2}$ where $C>0$ is a constant depending only on
the linear part of $g$. It is reasonable to expect a similar behavior of this function
in the case of Euclidean buildings. We show the following 
(this result was first conjectured by Rousseau, see \cite[4.8]{Rousseau:eximmob} 
and \cite[2.2]{RapoportZink}).

\begin{theorem}\label{mainthm}
Let $X$ be a Euclidean building without factors isometric to Euclidean spaces.
Let $g\in Isom(X)$ be an isometry of $X$. Then
$$d_g(x) \geq \sqrt{\delta_g^2 + C\cdot d(x,\minset_g)^2}$$
for a constant $C>0$ depending only on the spherical
Coxeter complex associated to $X$ and, if $g$ is hyperbolic, on
the type of the endpoint $c(\infty)$ of an axis $c$ of $g$.
\end{theorem}

Notice that the conclusion of Theorem~\ref{mainthm} can only be satisfied by
semisimple isometries (see Proposition~\ref{prop:parreau}), that is, Euclidean
buildings do not admit parabolic isometries.

\section{Preliminaries}

In this paper we consider Euclidean buildings from the $\CAT(0)$ viewpoint as 
presented in \cite[Section 4]{KleinerLeeb:quasi-isom},
we refer to it for the basic definitions and
facts about Coxeter complexes, Euclidean and spherical buildings.
For more information on $\CAT(0)$ spaces in general we refer to \cite{BridsonHaefliger}.

All geodesic segments, lines and rays will be assumed to be parametrized
by arc-length.

Let $(X,d)$ be a Euclidean building. 
Its Tits boundary $\tits X$ with the {\em Tits distance} $\dtits$ is a $\CAT(1)$
space admitting a unique structure (possibly trivial if $X$ is a Euclidean space)
as a {\em thick} spherical building modelled on a
spherical Coxeter complex $(S,W)$. 
We say that $(S,W)$ is the spherical Coxeter complex associated to $X$.
The {\em space of directions} or {\em link} $\Sigma_x X$ at a point $x\in X$ 
with the {\em angle metric} $\angle$ is also
a $\CAT(1)$ space admitting a natural structure as a spherical building modelled on the
same spherical Coxeter complex $(S,W)$, although this structure is not in general thick.
The spherical Coxeter complex $(S,W)$ splits off a spherical factor 
in its decomposition as spherical join (see \cite[Section 3.3]{KleinerLeeb:quasi-isom})
if and only if 
the Coxeter group $W$ has fixed points on the sphere $S$. 
Equivalently, if and only if
the {\em model Weyl chamber} $\weylch := S/W$ has diameter $>\pihalf$ (in this case,
it actually has diameter $\pi$). This spherical factor of the Coxeter complex
corresponds to a spherical factor of $\tits X$, which in turn corresponds to a
Euclidean factor of $X$ in its decomposition as a product of Euclidean buildings. 
An isometry $g$ of a Euclidean building $X=X'\times \R^k$
where $X'$ is a Euclidean building without Euclidean factors decomposes naturally
as $(g_1,g_2)$ where $g_1$ is an isometry of $X'$ and $g_2$ 
is an isometry of $\R^k$.
Since we can completely describe the displacement function of Euclidean isometries 
(cf.\ the paragraph previous to Theorem~\ref{mainthm} in the Introduction),
we restrict our attention to Euclidean buildings without Euclidean factors, 
or equivalently, with model Weyl chamber $\weylch$ of diameter $\leq \pihalf$.

For a spherical building $B$ modelled on the
spherical Coxeter complex $(S,W)$ 
there is a natural 1-Lipschitz map $\theta: B \rightarrow \weylch$, which
is an isometry on chambers. The image of a point in $B$ under this map
its called its {\em type}.
Let $\Gamma$ denote the isometry group of the model Weyl chamber $\weylch$.
If $(S,W)$ has no spherical factors, then $\weylch$ is a spherical simplex
and $\Gamma$ is finite.
In this case, the quotient $\weylch/\Gamma$ is again a spherical polyhedron. 
Let us denote the image of a point in $B$ under the composition 
$\tau:B\rightarrow \weylch\rightarrow \weylch/\Gamma$ by its {\em subtype}. 
Notice that antipodal points have the same subtype.

An isometry of a Euclidean building induces an isometry of its model Weyl chamber. 
Hence, isometries preserve the subtypes of directions and points at infinity.

If $g$ is a hyperbolic isometry of a $\CAT(0)$ space $X$, then $\minset_g$
is isometric to a product $Y\times \R$, where $Y$ is again a $\CAT(0)$ space.
The isometry $g$ acts on $Y$ as the identity and on $\R$ as a translation 
of length $\delta_g$. 
The sets $\{y\}\times\R \subset \minset_g$ are geodesics lines preserved by $g$
and are called {\em axes} of $g$. The axes are parallel to each other, thus 
$\minset_g$ is a subset of the parallel set of an axis.

The displacement function of an isometry 
$g$ of a $\CAT(0)$ space $(X,d)$ is a Lipschitz convex function. This kind of
functions on $\CAT(0)$ spaces are asymptotically linear along any ray. That
is, if $\xi\in\tits X$ and $\rho$ is a ray asymptotic to $\xi$, 
then $slope_g(\xi):=\lim_{t\to\infty}\frac{1}{t}d_g(\rho(t))$ 
exists and does not depend on the particular ray $\rho$.
Let now $\rho$ be a geodesic ray with 
$\rho(0)\in \minset_g$ and orthogonal to $\minset_g$.
The observation above implies that
there is a constant $K$($=slope_g(\rho(\infty))$) depending on $g$ and $\rho(\infty)$,
such that $d_g(\rho(t))\geq \frac{K}{2}t = \frac{K}{2}d(\rho(t),\minset_g)$ 
for $t$ big enough.
Proposition~\ref{prop:parreau} makes a more precise statement in
the case of Euclidean buildings and shows that
we can choose the constant $K$ to depend only on the associated spherical Coxeter complex.

Isometries of Euclidean buildings are always semisimple. This follows form the following
result of Parreau \cite[Theorem 4.1]{Parreau}. We include here a variant of her proof for
the convenience of the reader.

\begin{proposition}\label{prop:parreau}
Let $X$ be a Euclidean building without Euclidean factors and let $g$ be an isometry.
Let $D_a=D_a(g)=\{x\in X\;|\; d_g(x)\leq a\}$ denote the sublevel sets
of the displacement function. 
There is a constant $K>0$ depending only on 
the spherical Coxeter complex $(S,W)$ associated to $X$
such that if
$D_a\neq \emptyset$ then 
$$ d_g(x)+ a \geq Kd(x,D_a).$$
In particular, all isometries of $X$ are semisimple.
\end{proposition}

\begin{proof}
Suppose there is no such constant $K$. Then there exist sequences $X_n$
of Euclidean buildings with associated spherical Coxeter complex $(S,W)$,
isometries $g_n\in Isom(X_n)$, points $x_n\in X_n$ and real numbers $a_n$
such that $\emptyset \neq D_n:=D_{a_n}(g_n)\subset X_n$ and
$$d_{g_n}(x_n)+a_n \leq \frac{1}{n}d(x_n,D_n).$$

Let $y_n$ be the projection of $x_n$ onto $D_n$ and extend the segment
$y_nx_n$ to a ray $y_n\zeta_n$ with $\zeta_n\in\tits X_n$ (see Figure~\ref{fig:parreau}). 
After taking a 
subsequence, we may assume that the subtype of $\zeta_n$ converges to some
$\tau\in\weylch/\Gamma$. 
Since the distance between two points of the same subtype
can take only finitely many values, for $n$ big enough there is a unique point 
$\mu_n \in \tits X_n$ nearest to $\zeta_n$ of subtype $\tau$
and $\dtits(\mu_n,\zeta_n)\rightarrow 0$. 
This implies that $g_n\mu_n$ is also 
the unique point nearest to $g_n\zeta_n$ of subtype $\tau$.

Let $x_n'$ be the point on the ray $y_n(g_n\zeta_n)$ at distance $d(x_n,y_n)$
from $y_n$. Then $d(g_nx_n,x_n')\leq d(g_ny_n,y_n)\leq a_n$ because the rays
$y_n(g_n\zeta_n)$ and $(g_ny_n)(g_n\zeta_n)$ are asymptotic (see Figure~\ref{fig:parreau}). 
Hence
$$
\frac{1}{n}d(x_n,y_n)\geq d(x_n,g_nx_n)+a_n\geq d(x_n,g_nx_n)+d(g_nx_n,x_n')
\geq d(x_n,x_n').
$$
Considering the triangle $(y_n,x_n,x_n')$ 
this implies that $\angle_{y_n}(\zeta_n,g_n\zeta_n)=\angle_{y_n}(x_n,x_n')\rightarrow 0$.
Hence, from the triangle inequality in the link $\Sigma_{y_n}X_n$ we obtain
$$\angle_{y_n}(\mu_n,g_n\mu_n)\leq\angle_{y_n}(\mu_n,\zeta_n) + 
\angle_{y_n}(\zeta_n,g_n\zeta_n)
+\angle_{y_n}(g_n\zeta_n,g_n\mu_n) \rightarrow 0.
$$ 
It follows that for $n$ big enough $\angle_{y_n}(\mu_n,g_n\mu_n)=0$ because
$\mu_n$ and $g_n\mu_n$ have the same subtype.
Therefore the rays $y_n\mu_n$ and $y_n(g_n\mu_n)$ must initially coincide.
This in turn implies that for $z_n$ on the ray $y_n\mu_n$ close enough to $y_n$
(such that $z_n$ also lies on the ray $y_n(g_n\mu_n)$)
holds $d(z_n,g_nz_n)\leq d(y_n,g_ny_n)\leq a_n$ since the rays 
$y_n(g_n\mu_n)$ and $(g_ny_n)(g_n\mu_n)=g_n(y_n\mu_n)$ are asymptotic
(see Figure~\ref{fig:parreau}).
Thus $z_n\in D_n=\{x\in X_n\;|\; d_{g_n}(x)\leq a_n\}$ and since
$y_n$ is the projection of $x_n$ onto $D_n$ we get 
$\angle_{y_n}(x_n,z_n)\geq \pihalf$. 
We obtain a contradiction to
$\angle_{y_n}(x_n,z_n) = \angle_{y_n}(\zeta_n,\mu_n) \rightarrow 0$.
This proves the first part.

\begin{figure}\begin{center}
\includegraphics[scale=0.4]{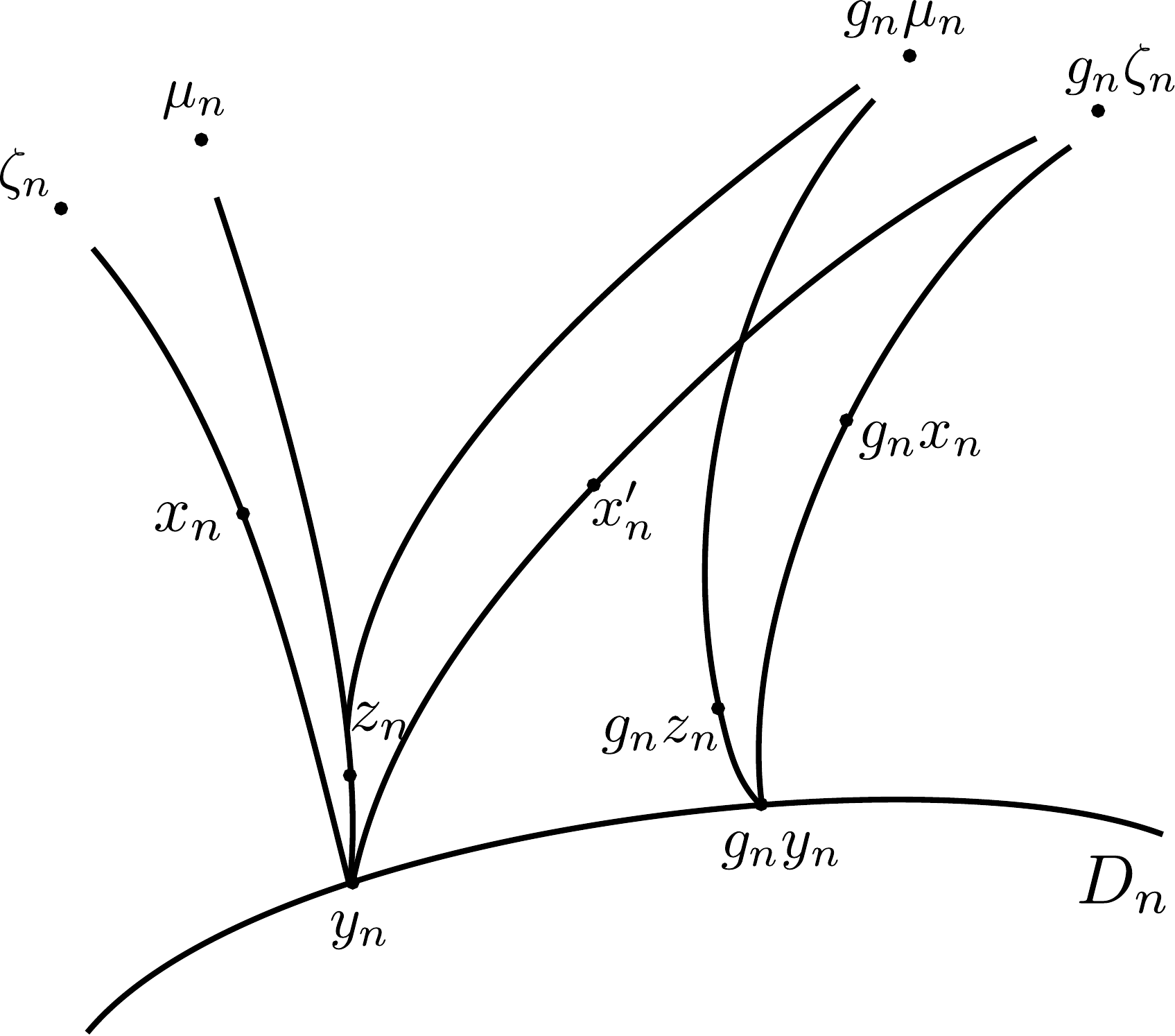}
\caption{Proof of Proposition~\ref{prop:parreau}}
\label{fig:parreau}
\end{center}
\end{figure}

The inequality for $d_{g}$ shows that for a given point $x\in X$ and 
$d_g(x)\geq a>\delta_g$ the distance of $x$ to $D_a$
is bounded independently of $a$. In other words, there is a $r>0$ such
that $B_x(r)\cap D_a\neq \emptyset$ for all $a>\delta_g$.
This implies that 
$B_x(r)\cap\minset_g=B_x(r)\cap(\bigcap_a D_a)$ is nonempty and therefore $g$ is semisimple.
\end{proof}

\begin{remark}
 The lower bound for $d_g$ given by Proposition~\ref{prop:parreau} is
interesting only for big values of $d(x,\minset_g)$:
if $\delta_g>0$ and $d(x,\minset_g)\leq \delta_g/C$ we just get the trivial 
bound $d_g(x)\geq 0$.
\end{remark}

We will also need the following generalization of \cite[Lemma 4.6.3]{KleinerLeeb:quasi-isom}.

\begin{proposition}\label{prop:rayinparalset}
Let $c$ be a geodesic line in the Euclidean building $X$ with 
$c(\infty)=:\xi\in\tits X$
and denote with $P_c$ the parallel set of $c$.
Let $\alpha>0$ be the distance of $\xi$ to the boundary of the union of all chambers
in $\tits X$ containing $\xi$ 
(notice that $\alpha$ depends only on the (sub)type of $\xi$ in $\weylch$). 
Let $\rho$ be a geodesic ray asymptotic to $\xi$. 
Then $\rho(t)\in P_c$ for $t\geq \frac{d(\rho(0),P_c)}{\sin\alpha}$.
\end{proposition}

\begin{proof}
Let $A\subset X$ be an apartment containing the ray $\rho$. The convex hull
in $\tits X$ of $c(-\infty)\in\tits X$ and all chambers in $\tits A$ containing $\xi$ is an
apartment contained in $\tits P_c$, which is the boundary at infinity of an apartment
$A'\subset P_c$. It follows that $\tits A \cap \tits A'$ is a union of chambers containing
$\xi$ in its interior. 
\cite[Lemma 4.6.3]{KleinerLeeb:quasi-isom} applied to the regular geodesic lines in
$A$ asymptotic to points in $\tits A \cap \tits A'$ implies that
$A\cap A'$ is a convex subset with boundary $\tits(A\cap A') = \tits A \cap \tits A'$
(cf.\ \cite[Lemma 4.6.5]{KleinerLeeb:quasi-isom}).
In particular, $A\cap A'$  contains a subray of $\rho$.
That is, $\rho(t)\in P_c$ for $t$ big enough.

Let $t_0$ be such that $\rho(t_0)$ is the first point of $\rho$ in $P_c$. We want to show
that $t_0\leq \frac{d(\rho(0),P_c)}{\sin\alpha}$.
Let $x=\rho(0)$, $y=\rho(t_0)$ and let $z$ be the projection of $x$ onto $P_c$.
Then $\angle_z(x,y)\geq \pihalf$.
Since $y$ is the point where $\rho$ enters $P_c$ we see that 
$\angle(\ora{yx},\Sigma_y P_c)>0$. 
In particular, the chambers in $\Sigma_y X$ containing $\ora{yx}$ are not
contained in $\Sigma_y P_c$. 
It follows that $\angle(\ora{yx},\Sigma_y P_c)\geq \alpha$
because $\ora{yx}$ is antipodal to $\ora{y\xi}$ and therefore 
both have the same subtype as $\xi$.
This in turn implies that $\angle_y(x,z)\geq \alpha$.
By considering a comparison triangle for $(x,y,z)$ we obtain 
$d(x,y)\sin\alpha \leq d(x,z)$, that is, $t_0\sin\alpha\leq d(\rho(0),P_c)$.
\end{proof}

\begin{remark}
The non-existence of parabolic isometries of Euclidean buildings can also be deduced
from Proposition~\ref{prop:rayinparalset}
as follows. By \cite[Corollary 1.5]{LytchakCaprace:atinfinity} an isometry $g$
has a fixed point in $X\cup\tits X$. If $g$ is not elliptic, it must fix
a point $\xi\in \tits X$ at infinity.
Let $c$ be a geodesic line asymptotic to $\xi$. 
The parallel set of $c$ splits as a product
$P_c\cong Y\times \R$ where $Y\subset X$ 
is again a Euclidean building with $\rank(Y)=\rank(X)-1$.
The isometry $g$ induces an isometry $\bar g $ of $Y$ as follows:
for $y\in Y$ consider the ray $g(\bar y\xi)$
where $\bar y:=(y, 0)$. Since $\xi$ is fixed by $g$, this ray is again 
asymptotic to $\xi$. Therefore by Proposition~\ref{prop:rayinparalset} it
eventually coincides with a line parallel to $c$ of the form $\{y'\}\times\R$. 
The map $y\mapsto y'=:\bar g(y)$ is an isometry of $Y$.
By induction on the rank of the building $\bar g$ must be a semisimple isometry.
If $\bar g$ is elliptic and $y\in Y$ is a fixed point, then the rays $\bar y\xi$
and $g(\bar y\xi)$ must eventually coincide. It follows that there is a geodesic
line which eventually coincides with $\bar y\xi$ that is translated by $g$,
thus this line is an axis of $g$ and 
we conclude that $g$ is semisimple. 
If $\bar g$ is hyperbolic and
$l\subset \minset_{\bar g}\subset Y$ is an axis of $\bar g$, then
the rays $(g\bar y)\xi$ for $y\in l$ eventually
enter the two dimensional flat $F=l\times\R$ because $l$ is preserved by $\bar g$. 
It follows that for $x=(y,t)\in F$ with $t$ big enough $x,gx,g^2x\in F\cap gF$
and these three points lie on a segment because $y,\bar gy,\bar g^2y\in l$.
This implies that there is a geodesic line (not asymptotic to $\xi$) 
through $x,gx,g^2x$  
that is translated by $g$ and again $g$ is semisimple.
\end{remark}

\section{Proof of the main result}

If $g$ is elliptic, then the result follows from Proposition~\ref{prop:parreau}.

Let $g$ be a hyperbolic isometry with axis asymptotic to $\xi\in\tits X$.
That is, if $c$ is an axis of $g$, then $g(c(t)) = c(t+\delta_g)$ and $c(\infty)=\xi$.
Let $P\cong Y\times\R$ be the parallel set 
 of an axis of $g$. 
The metric space $Y$ is again an Euclidean building and we denote again with $d$
its induced distance. 
Its Tits boundary $\tits Y$ can be canonically identified with $\Sigma_\xi \tits X$.
It follows that the spherical Coxeter complex associated to $Y$
depends only on $(S,W)$ and the type of the face $\sigma\subset \tits X$ 
containing $\xi$ in its interior.
$Y$ splits off a Euclidean factor of dimension $\dim(\sigma)$.
The isometry $g$ restricts to an isometry of $P$ and
$\minset_g\cong Y'\times\R$ with $Y'\subset Y$.
Thus, $g$ induces an elliptic isometry $\bar g$ of $Y$ with fixed point set $Y'$.
The projection of $\bar g$ to the Euclidean factor of $Y$ induces an isometry
of the face $\sigma$. Therefore there are only finitely many possibilities for
the linear part of the projection of $\bar g$ to the Euclidean factor
and they depend only on the type of $\sigma$.
Since there are also only finitely many types of faces of $\tits X$, 
Proposition~\ref{prop:parreau} gives us a constant $C'>0$ 
depending only on $(S,W)$  such that 
$d(y,\bar gy)\geq C'd(y,Y')$ for all $y\in Y$.

Suppose first that $x\in P\cong Y\times\R$ (cf.\ \cite[Prop. 4.4]{Rousseau:eximmob}).
Then $x=(y,t)\in Y\times\R$ and $gx=(\bar gy,t+\delta_g)$. It follows that
$$d(x,gx)^2= d(y,\bar gy)^2+\delta_g^2 \geq C'^2d(y,Y')^2+\delta_g^2=
C'^2d(x,\minset_g)^2+\delta_g^2.$$
This proves the assertion in this case. We want now to generalize this idea.

Let $b_\xi(x) = \lim\limits_{t\rightarrow\infty} (d(x,c(t))-t)$ 
be a Busemann function centered at $\xi$,
where $c$ is an axis of $g$ with $c(\infty) = \xi$. 
For $x\in X$ holds 
$b_\xi(gx) = \lim\limits_{t\rightarrow\infty} (d(gx,c(t))-t)=  
\lim\limits_{t\rightarrow\infty} (d(gx,c(t+\delta_g))-(t+\delta_g))=
\lim\limits_{t\rightarrow\infty} (d(gx,gc(t))-t) - \delta_g =
\lim\limits_{t\rightarrow\infty} (d(x,c(t))-t) - \delta_g = b_\xi(x) - \delta_g$.
It follows that $g$ maps
the horosphere $Hs(r):=b_{\xi}^{-1}(r)$ (the horoball $Hb(r):=b_{\xi}^{-1}((-\infty,r])$)
to the horosphere $Hs(r-\delta_g)$ 
(the horoball $Hb(r-\delta_g)$).

Consider a point $x\in X$. Let $x'$ be the point 
on the ray $x\xi$ at distance $\delta_g$ from $x$. Then $x'$ and $gx$ are
in the same horosphere $Hs(b_\xi(gx))$ and $x'$ is the projection of $x$
onto the horoball $Hb(b_\xi(gx))$. It follows that
$\angle_{x'}(x,gx)\geq \pihalf$ and from triangle comparison 
with respect to the triangle $(x,x',gx)$ we conclude
$$d(x,gx)^2\geq d(x,x')^2 + d(x',gx)^2 = \delta_g^2+d(x',gx)^2.$$
Hence, it suffices to show that there is a constant $C>0$
depending only on the spherical
Coxeter complex associated to $X$ and the type of $\xi$ such that
$d(x',gx) \geq Cd(x,\minset_g)$.
In the case $x=(y,t)\in P\cong Y\times\R$ above, we have $x'=(y,t+\delta_g)$. Hence,
$d(x',gx) = d(y,\bar g y) \geq C'd(y,Y') = C'd(x,\minset_g)$.

Suppose there is no such constant $C>0$.
Then there exist sequences $X_n$
of Euclidean buildings with associated spherical Coxeter complex $(S,W)$,
hyperbolic isometries $g_n\in Isom(X_n)$ with axes asymptotic to a
point $\xi_n\in\tits X_n$ of constant
subtype $\tau\in\weylch/\Gamma$, and points $x_n\in X_n$ 
such that
$$d(x_n',g_nx_n) \leq \frac{1}{n}d(x_n,M_n),$$
where $M_n:=\minset_{g_n}$ and
$x_n'$ is the projection of $x_n$ onto
the horoball $Hb(b_{\xi_n}(g_nx_n))$ centered at $\xi_n$.

After rescaling the buildings $X_n$ we may assume that $\delta_{g_n}=\delta$
is independent of $n$. From Proposition~\ref{prop:parreau} we get a constant 
$K>0$ depending only on $(S,W)$ such that
$$Kd(x_n,M_n)-\delta \leq d(x_n,g_nx_n)\leq d(x_n,x_n')+d(x_n',g_nx_n)\leq
\delta + \frac{1}{n}d(x_n,M_n).$$
This implies that $d(x_n,M_n)\leq\frac{2\delta}{K-1/n}\leq \frac{3\delta}{K}$
for $n$ big enough.

Let $P_n\subset X_n$ the parallel set of an axis of $g_n$.
Proposition~\ref{prop:rayinparalset} implies that there is an $\alpha$ independent
of $n$ such that a point on the ray $x_n\xi_n$
at distance $\geq \frac{d(x_n,P_n)}{\sin\alpha}$ from $x_n$ lies in $P_n$.
For $k\in\N$, let $x_n^k$ be the point on the ray $x_n\xi_n$
at distance $k\delta$ from $x_n$ (see Figure~\ref{fig:mainthm}). 
Choose a fixed $m\in \N$ such that $m\geq \frac{3}{K\sin\alpha}$. 
Then $x_n^m$ lies in $P_n$ for all $n$ big enough because 
$m\delta\geq \frac{3\delta}{K\sin\alpha} \geq 
\frac{d(x_n,M_n)}{\sin\alpha}\geq \frac{d(x_n,P_n)}{\sin\alpha}$.

We show now inductively on $k$ that $d(g_n^k x_n,x_n^k)\leq \frac{k}{n}d(x_n,M_n)$.
For $k=0$ this is trivial. Let us prove it for $k+1$.
First observe that the rays $x_n^k\xi_n$ and $(g_n^kx_n)\xi_n=g_n^k(x_n\xi_n)$ are
asymptotic and therefore 
$d(x_n^{k+1},g_n^kx_n')\leq d(x_n^k,g_n^kx_n)$ (see Figure~\ref{fig:mainthm}). Hence,
\begin{align*}
d(g_n^{k+1} x_n,x_n^{k+1})&\leq d(g_n^{k+1} x_n, g_n^k x_n') +d(g_n^k x_n',x_n^{k+1})\\
&\leq d(g_n x_n, x_n') + d(x_n^k,g_n^kx_n)\\ 
&\leq \frac{1}{n}d(x_n,M_n) +\frac{k}{n}d(x_n,M_n) = \frac{k+1}{n}d(x_n,M_n).
\end{align*}

\begin{figure}\begin{center}
\includegraphics[scale=0.8]{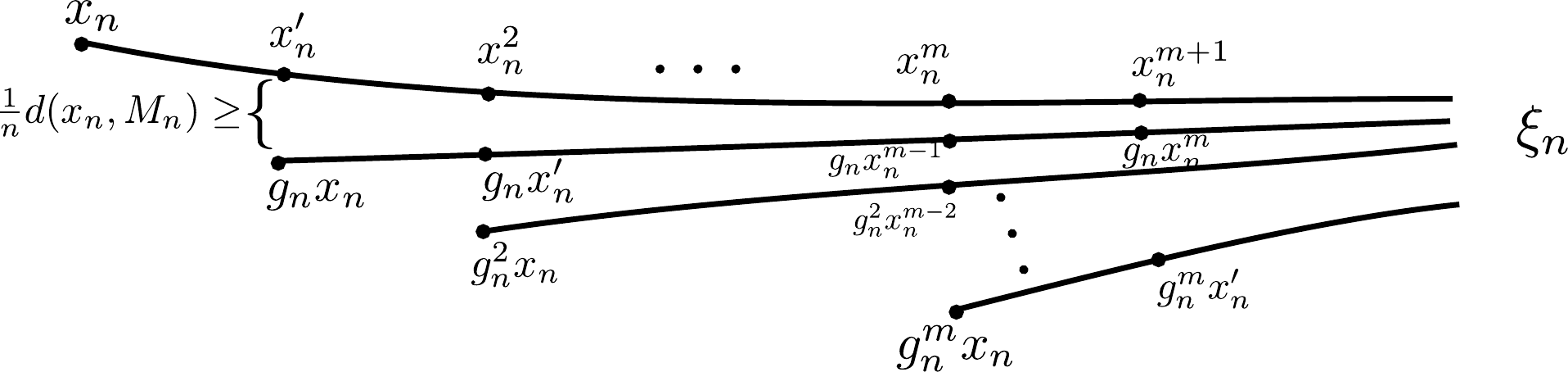}
\caption{Proof of the main result}
\label{fig:mainthm}
\end{center}
\end{figure}

In particular, we have that $d(x_n^m,g_n^m x_n)\leq\frac{m}{n}d(x_n,M_n)$ and this implies 
$$d(x_n^m,M_n)\geq d(g_n^m x_n,M_n) - d(x_n^m,g_n^m x_n)  \geq
d(x_n,M_n) - \frac{m}{n}d(x_n,M_n)=\left(1-\frac{m}{n}\right)d(x_n,M_n).$$

Since $x_n^m$ lies in the parallel set $P_n$ of an axis of $g_n$, we already know
in this case that there is a constant $C'>0$ such that
$d(x_n^{m+1},g_nx_n^m) \geq C'd(x_n^m,M_n) \geq C'\left(1-\frac{m}{n}\right)d(x_n,M_n)$.

On the other hand, since the rays $x_n'\xi_n$ and $(g_nx_n)\xi_n = g_n(x_n\xi_n)$ 
are asymptotic, we see that $d(x_n',g_nx_n)\geq d(x_n^{m+1},g_nx_n^m)$. Thus, we obtain
$$
\frac{1}{n}d(x_n,M_n)\geq 
d(x_n',g_nx_n)\geq d(x_n^{m+1},g_nx_n^m) \geq C'\left(1-\frac{m}{n}\right)d(x_n,M_n).
$$
This implies $1\geq C'(n-m)$, which is a contradiction for $n$ big enough.
\qed

\bigskip
{\em Acknowledgments.} 
The results on this paper were obtained during a stay at the Max Planck Institute
for Mathematics in Bonn in 2011. We are grateful to the MPI
for financial support and for its great hospitality.

\bibliography{MyBibliography}
\bibliographystyle{amsalpha}

\noindent {\small
\textsc{Mathematisches Institut, Universit\"at M\"unchen, Theresienstr. 39, 
D-80333 M\"unchen, Germany}\\
{\em E-mail:} {\texttt cramos@mathematik.uni-muenchen.de}}

\end{document}